\documentclass[a4paper,12pt]{amsart}

\usepackage{amssymb,amsbsy,amsmath,amsfonts,amssymb,amscd}
\usepackage{latexsym}
\usepackage{geometry}

\input xy
\xyoption{all}

\newcommand\ga{\gamma}

\newcommand{\CC}{\ensuremath{\mathbb{C}}}

\newcommand{\ZZ}{\ensuremath{\mathbb{Z}}}
\newcommand{\QQ}{\ensuremath{\mathbb{Q}}}

\newcommand{\NN}{\ensuremath{\mathbb{N}}}
\newcommand{\hol}{\ensuremath{\mathcal{O}}}

\newcommand{\PP}{\ensuremath{\mathbb{P}}}

\newcommand{\ra}{\ensuremath{\rightarrow}}

\def\eea{\end{eqnarray*}}
\def\bea{\begin{eqnarray*}}

\newcommand\dual{\mathrel{\raise3pt\hbox{$\underline{\mathrm{\thinspace d
\thinspace}}$}}}
\newcommand\qe{\ifhmode\unskip\nobreak\fi\quad $\Box$}       % box for QED

\def\BOX{\hfill\lower.5\baselineskip\hbox{$\Box$}}
\newtheorem{theo}{Theorem}[section]
\newtheorem{remarkk}[theo]{Remark}
\newenvironment{rem}{\begin{remarkk}\rm}{\end{remarkk}}

\newtheorem{defin}[theo]{Definition}
\newenvironment{definition}{\begin{defin}\rm}{\end{defin}}

\newtheorem{prop}[theo] {Proposition}
\newtheorem{cor}[theo]{Corollary}
\newtheorem{lemma}[theo]{Lemma}
\newtheorem{example}[theo]{Example}

\newtheorem{conj}[theo]{Conjecture}
\newtheorem{problem}[theo]{Problem}
\newtheorem{question}[theo]{Question}

\newcommand{\BT}{\ensuremath{\mathbb{T}}}

\DeclareMathOperator{\Sing}{Sing}
\DeclareMathOperator{\lcm}{lcm}

\begin{document}

\title[Product-Quotient Surfaces]{Product-Quotient Surfaces: Results and Problems}
\author{I. Bauer}

\thanks{The present survey is an extended version of a talk given at the Kinosaki Symposium on Algebraic Geometry 2011 reviewing several joint articles with F. Catanese, F. Grunewald and R. Pignatelli, as well as some yet unpublished results. All the research contained in these papers took place in the realm of the DFG Forschergruppe 790
"Classification of algebraic surfaces and compact complex manifolds".}

\date{\today}
\maketitle
\tableofcontents
\section*{Introduction}

The following is an extended version of a talk given at the Kinosaki Symposium on Algebraic Geometry in October 2011. The aim is to give an overview of product-quotient surfaces, the results that have been proven so far in collaboration with several different authors as well as pointing out some of the the problems that are still open.

In what follows we will use the basic notations from the classification theory of complex projective surfaces, in particular
the basic numerical invariants $K_S^2$, $p_g :=h^0(S, \Omega^2_S)$, $q(S) :=h^1(S, \hol_S)$; 
the reader unfamiliar with these
may consult  e.g. \cite{beau}.

By Gieseker's theorem (cf. \cite{gieseker}) and standard inequalities (cf. \cite[thm. 2.3 and the 
following discussion]{surveypg=0})  minimal surfaces of general type with $p_g = 0$
 yield  a
finite number of irreducible components of the
moduli space of surfaces of general type. 

Even if surfaces of general type with $p_g=0$  are the surfaces of general type which achieve
the minimal value $1$ for the
holomorphic  Euler-Poincar\'e characteristic $\chi(S) := p_g(S) -q(S)
+1$, the guess that
they should be ``easier'' to understand than  other surfaces with
higher invariants is false by all means.

We refer to Mumford's provocative question, which he posed after in the previous decade it became clear that constructing surfaces of general type with $p_g=0$ is not easy at all. We refer to table 1 in \cite{bpv} to confirm that there were quite few examples known even in 2004.

\begin{question}[1980, Montreal]
 Can a computer classify smooth complex projective surfaces with $p_g=0$?
\end{question}
Our approach will be very much in the spirit of Mumford's question. Product-quotient surfaces are surfaces which can be constructed with the help of a computer algebra program as MAGMA, even a classification up to a certain point can be obtained by a computer. Still one quickly arrives at a limit where certain geometric features cannot be detected by a computer anymore.

We would like to refer the reader to a recent survey on surfaces of general type with $p_g = 0$
\cite{surveypg=0}, for a historical account and an update on what is known.

The paper is organized as follows: in the first two sections  we define what are product-quotient surfaces,recall their combinatorial properties and explain the necessary results to obtain the algorithm developed in \cite{bp}.

In the third chapter we summarize the results on the classification of product-quotient surfaces of general type with $p_g=q=0$ that were obtained in a series of papers (cf. \cite{bacat}, \cite{pg=q=0}, \cite{4names}, \cite{bp}). 
In the above articles a complete classification of product-surfaces $S$ of general type with $p_g= q=0$, under the additional hypothesis that $S$ is minimal,  is given. It is in general a highly non trivial problem to decide whether a given product-quotient surface (where the singular model has non-canonical singularities) is minimal or not. Chapter 4 deals with the problem of finding rational curves on product-quotient surfaces, and exhibits a non-minimal example (``the fake Godeaux surface'', discussed at length in \cite{bp}), where its minimal model is computed.
The last section is dedicated to the problems which are still open and prevent us still from finishing the classification of product-quotient surfaces of general type with $p_g=0$ as well as  some yet unpublished partial results (in collaboration with R. Pignatelli) which go in the direction of a solution.

\section{What are product-quotient surfaces}

We consider the following situation:
let $G$ be a finite group acting on two compact Riemann surfaces $C_1$, $C_2$ of respective genera $g_1, g_2 \geq 2$.
We shall consider the diagonal action of $G$ on $C_1 \times C_2$ and
in this situation we say for short: the action of $G$ on
$C_1 \times C_2$ is {\em unmixed}. By \cite{FabIso} we may assume w.l.o.g.
that $G$ acts faithfully on both factors.
\begin{defin}\label{prodquot}
The minimal resolution $S$ of the singularities of $X = (C_1 \times C_2)/ G$, where $G$ is a finite group with an unmixed
action on the direct product of two compact Riemann surfaces $C_1$, $C_2$ of respective genera at least
two, is called a {\em product-quotient surface}.

$X$ is called  {\em the quotient model} of the product-quotient surface.
\end{defin}

\begin{rem} \
\begin{enumerate}
 \item  Note that there are finitely many points on $C_1 \times C_2$ with non trivial stabilizer, which is automatically cyclic. Hence the quotient surface $X:= (C_1 \times C_2) / G$ has
a finite number of cyclic quotient singularities. 

Recall that every cyclic
quotient singularity is locally analytically isomorphic to the quotient of $\mathbb{C}^2$ by
      the action of a diagonal linear automorphism with eigenvalues
      $\exp(\frac{2\pi i}{n})$, $\exp(\frac{2\pi i a}{n})$ with $g.c.d(a,n) = 1$; this is called a 
{\em singularity of type $\frac{1}{n}(1,a)$}. Two singularities of respective types $\frac{1}{n}(1,a)$ and 
$\frac{1}{n'}(1,a')$ are locally analytically isomorphic if and only if $n=n'$ and either $a=a'$ or $aa'
\equiv 1 \mod n$.
\item It is well known that the exceptional divisor $E$ of the minimal resolution of a cyclic quotient singularity of type  $\frac{1}{n}(1,a)$ is a {\em Hirzebruch-Jung string}, i.e., $E = \bigcup_{i=1}^l E_i$ where all $E_i$ are smooth rational curves, $E_i^2 = -b_i$, $E_i\cdot E_{i+1}
= 1$ for $i \in\{1, \ldots , l-1\}$ and $E_i\cdot E_j=0$ otherwise. The $b_i$ are given by the formula
$$
\frac{n}{a} = b_1 - \frac{1}{b_2 - \frac{1}{b_3 - \ldots}}.
$$
\item We denote by $K_X$ the canonical (Weil) divisor on the normal surface $X$
corresponding to $i_* ( \Omega^2_{X^0})$, $ i\colon X^0 \ra X$ being the
inclusion of the
smooth locus of $X$. According to Mumford we have an intersection
product with values
in $\QQ$ for Weil divisors on a normal surface, and in particular we may consider
the self intersection  of the canonical divisor,
$$
K_X^2 =
\frac{8 (g(C_1) - 1) (g(C_2) - 1)}{|G|}
\in
      \mathbb{Q},
$$
 which is not necessarily an integer.

Moreover, we have (in a neighborhood of $x$)
$$ K_S = \pi^* K_X + \sum_{i=1}^l a_i E_i,
$$  where the rational numbers $a_i$ are determined by the conditions
    $$(K_S + E_j)E_j =
-2, \ \ \
(K_S -
\sum_{i=1}^l a_iE_i)E_j = 0, \ \
\forall j= 1, \dots ,l.$$
\item Note that $S$ is in general not a minimal model. In fact, to handle this is one of the biggest difficulties when trying to push the classification to the most general case, i.e., admitting no artificial restrictions on the singularities of $X$.
\item Since the minimal resolution  $\pi \colon S
\rightarrow X$ of the
singularities of $X$ replaces each singular point by a tree of smooth
rational curves, we have, by van Kampen's theorem, that
$\pi_1(X) = \pi_1(S)$. A presentation of this fundamental group can easily be given using a result of Armstrong (cf. \cite{armstrong1}, \cite{armstrong2}). Unfortunately a presentation of a group might not give much information about the group (e.g. it is in general undecidable whether the group is trivial or not). Using a structure theorem for the fundamental group of a product-quotient surface proved in \cite{4names}, we can use the fundamental group of $X$ (or $S$) as invariant to distinguish different connected components of the moduli space of surfaces of general type.

\end{enumerate}

\end{rem}
We will need only the following combinatorial information about the singular locus $\Sing(X)$ of the quotient model $X$ of a product-quotient surface:
\begin{defin} \
 \begin{enumerate}
  \item Let $X$ be a normal complex surface and suppose that the singularities of $X$ are cyclic quotient
singularities. Then we define  a  {\em representation of the basket of singularities of $X$}  to be  
a multiset
$$
\mathcal{B}(X) := \left\{\lambda \times \left(\frac{1}{n} (1,a)\right) : X \ {\rm has \ exactly} \ \lambda \
{\rm singularities }  \ {\rm  of \ type} \ \frac{1}{n} (1,a) \right\}.
$$

I.e., $\mathcal{B}(X) = \{2 \times \frac 13 (1,1), \frac 14 (1,3) \}$ means that the singular 
locus of $X$ consists of two $\frac 13 (1,1)$-points and one $\frac 14 (1,3)$-point.
\item Consider the set of multisets of the form 

$$
\mathcal{B} := \left\{\lambda \times \left(\frac{1}{n} (1,a)\right) : a, n,   \lambda \in \mathbb{N}, \ 
a < n, \ \gcd(a,n)=1 \right\},
$$
and consider the equivalence relation generated by  "$\frac{1}{n}(1,a)$ is equivalent to $\frac1n(1,a')$", where $a' = a^{-1}$ in
$(\ZZ / n \ZZ)^*$. A {\em basket of singularities} is then an equivalence class.
 \end{enumerate}

\end{defin}

The invariants of the basket of singularities, which are used in \cite{bp}, are the following:

\begin{defin}\label{oldinv}
Let $x$ be a singularity of type $\frac 1n (1,a)$ with
$\gcd(n,a) = 1$ and let $1 \leq a' \leq n-1$ such that $a' = a^{-1}$ in
$(\ZZ / n \ZZ)^*$. Moreover, write $\frac na$ as a continued fraction:
$$
\frac{n}{a} = b_1 - \frac{1}{b_2 - \frac{1}{b_3 - \ldots}} =:[b_1, \ldots, b_l].
$$
Then we define the following correction terms:

\begin{itemize}
 \item[i)] $k_x := k(\frac 1n (1,a)):= -2 + \frac{2+a+a'}{n} + \sum(b_i-2) \geq 0$;
\item[ii)] $e_x := e(\frac 1n (1,a)):= l + 1 - \frac 1n \geq 0$;
\item[iii)] $B_x := 2e_x + k_x$.
\end{itemize}

 Let ${\mathcal B}$ be the basket of singularities of $X$ (recall that $X$ is normal and has only cyclic quotient singularities).
Then we use the following notation
$$k({\mathcal B}):=\sum_{x \in {\mathcal B}} k_x, \ \ \ 
e({\mathcal B}):=\sum_{x \in {\mathcal B}} e_x, \ \ \ B({\mathcal B}):=\sum_{x \in {\mathcal B}} B_x.
$$
\end{defin}

With these invariants $K_S^2$ and $e(S)$ of a product-quotient surface can be expressed as follows.

\begin{prop}[\cite{4names}, prop. 2.6, and \cite{polmi}, cor. 3.6]\label{k2e}
Let $S \ra X:=( C_1 \times C_2 )/G$ be the minimal resolution of singularities of $X$. Then we have the following two formulae for the self intersection of the canonical divisor of $S$ and the topological Euler characteristic of $S$:
\begin{equation*}
 K_S^2 = \frac{8 (g_1 -1)(g_2-1)}{|G|} - k({\mathcal B});
\end{equation*}
\begin{equation*}
e(S) = \frac{4 (g_1 -1)(g_2-1)}{|G|} + e({\mathcal B}).
\end{equation*}
\end{prop}

A direct consequence of the above is the following:
\begin{cor}\label{k2basket}
Let $S \ra X:=( C_1 \times C_2 )/G$ be the minimal resolution of singularities of $X$. Then
$$
K_S^2 = 8 \chi(S) - \frac 13 B({\mathcal B}).
$$
\end{cor}

\begin{proof}
 By prop. \ref{k2e} we have 
$$
e(S) = \frac{K_S^2 + B(\mathcal{B})}{2}.
$$
By Noether's  formula we obtain
$$
12\chi(S)=K^2_S+e(S) = \frac{3K_S^2 + B(\mathcal{B})}{2}
$$
\end{proof}

\section{The combinatorics of product-quotient surfaces}

The key point for studying product-quotient surfaces  is the fact that the geometry of the surface $S$ is encoded in the combinatorics of the finite group $G$.

We start with the following observations.\begin{rem}\label{general}
1) Let $S$ be a surface of general type. Then 
$p_g(S) \geq q(S) := h^1(S, \hol_S)$. In particular, $p_g = 0$ implies $q=0$. If $S$ is minimal, then $K_S^2 > 0$.

2) Let $S$ be a product-quotient surface with quotient model $X= (C_1 \times C_2)/G$. 
If $q(S) = 0$, then  $C_i/G \cong \PP^1$. If $S$ is of general type, then $g(C_i) \geq 2$.
\end{rem}

Since here we are interested only in {\em regular} surfaces (i.e., surfaces $S$ with $q(S) =0$), we only need to recall the definition of a special case of an orbifold surface group: a polygonal group, 
(cf. \cite{4names} for the general situation).

\begin{defin}\label{polgr}
A {\em polygonal group} of {\em signature} $(m_1, \dots m_r)$ is given by the following presentation:

\begin{equation*}
\mathbb T (m_1, \ldots ,m_r) :=
\langle c_1,\ldots, c_r |
    c_1^{m_1},\ldots ,c_r^{m_r}, c_1\cdot
\ldots \cdot c_r \rangle.
\end{equation*}

\end{defin}

Let $p, p_1, \dots , p_r \in \PP^1$ be $r+1$ different points and for each $1 \leq i \leq r$ choose 
a simple
geometric loop
$\gamma_i$  in $\pi_1(\PP^1 \setminus \{p_1, \dots , p_r \}, p)$ 
around $p_i$, such that $\gamma_1 \cdot \ldots \cdot \gamma_r = 1$.

Then $\mathbb T (m_1, \ldots ,m_r)$ is the factor
group
of $\pi_1(\PP^1 \setminus \{p_1, \dots , p_r \}, p)$
 by the subgroup normally generated by $\ga_1^{m_1}, \dots , \ga_r^{m_r}$.

Hence, by {\em Riemann's existence theorem}, any curve $C$ together with an action of a finite group $G$ on it
such that $C/G\cong \PP^1$ is determined (modulo automorphisms) 
by the following data:

1) the branch point set $\{p_1, \dots p_r \} \subset \PP^1$;

2) the kernel of the monodromy homomorphism $\pi_1 ( \PP^1 \setminus \{p_1,
\dots p_r \},p) \rightarrow G$ which, once chosen loops $\gamma_i$ as above, factors through 
$\BT (m_1, \ldots ,m_r)$, where $m_i$ is the branching index of $p_i$; therefore 
giving the monodromy homomorphism  is equivalent to give

2') an {\em appropriate orbifold} homomorphism
$$
\varphi \colon \BT (m_1, \ldots ,m_r) \rightarrow G,
$$
i.e., a  surjective homomorphism  such that
 $\varphi (c_i) $ is an element of order exactly $m_i$, 
with the property that
 the {\em Hurwitz' formula} for the genus $g$ of $C$
holds:
$$ 2g - 2 = |G|\left(-2 + \sum_{i=1}^r \left(1 -
\frac{1}{m_i}\right)\right).
$$

\begin{definition}
An $r$-tuple $(g_1, \ldots , g_r)$ of elements of a group $G$ is called a {\em spherical system of 
generators of $G$},
if $<g_1, \ldots , g_r> =G$ and $g_1 \cdot \ldots \cdot g_r =1$. 
 \end{definition}

Note that $(\varphi(c_1), \ldots , \varphi(c_r))$ is a spherical system of generators of $G$. Vice versa, a spherical system of 
generators of $G$ determines a polygonal group $\BT$ together with an appropriate orbifold homomorphism
$\varphi \colon \BT \rightarrow G$.

\noindent
Therefore a product-quotient surface $S$ of general type with $p_g = 0$
determines the following data 
\begin{itemize}
\item a finite group $G$;
\item two sets of points $\{p_1, \ldots, p_r\}$ and $\{q_1, \ldots, q_s\}$ in $\PP^1$;
\item two spherical systems of generators of $G$ of respective length $r$ and $s$.
\end{itemize}

Vice versa, the data above determine the product-quotient surface. 

\begin{rem}
 Different data may determine the same surface. This is in fact solved by considering the orbits of an action of a product of braid groups on the pairs of systems of spherical generators of a fixed group $G$ (cf. e.g. \cite{bp}).
\end{rem}

In order to get an algorithm that computes all product-quotient surfaces $S$ with fixed $\chi(S)$ and $K_S^2$ we need the following: 

\begin{lemma}
There are positive numbers  $D$, $M$, $R$, $B$, which depend explicitly (and only) on the basket of singularities singularities of $X$ such that:
\begin{enumerate}
\item  $K_S^2 = 8\chi - B$;
\item $r, s \leq R$, such that  $\forall \ i, j:  \ m_i, n_j \leq M$;
\item $|G| = \frac{K_S^2 + D}{2(-2+\sum_1^r(1-\frac{1}{m_i}))(-2+\sum_1^s(1-\frac{1}{n_i}))}$.
\end{enumerate}
\end{lemma}

From this we can deduce an algorithm as follows: 

\begin{enumerate}
\item Fix integers $\chi$ and $K^2$;
\item find all possible baskets $\mathcal{B}$ of singularities with $B(\mathcal{B}) = 8 - K^2$;
\item for a fixed basket find all signatures $(0;m_1, \ldots , m_r)$ satisfying $2)$;
\item for each pair of signatures check all groups of order as in $3)$, whether there is a surjective homomorphism ${\mathbb T}(0;m_1, \ldots , m_r) \rightarrow G$;
\item check the singularities of the surfaces in the output.
\end{enumerate}

Still to implement the algorithm in a way to be able to push through the computations we have to face several problems.

\begin{problem} \
\begin{itemize}
\item[i)] We have to search groups of a given order for generators, which have to fulfill certain conditions (orders as  512, 1024, 1536,...). These cases in fact have to be excluded in the general run of the program, and have to be treated by hand separately.
\item[ii)] The algorithm is very time and memory consuming especially for $K^2$ small compared to $\chi$. We have proved and implemented stronger estimates and conditions on $\Sing(X)$ and the possible signatures which allows us to get a complete list for $\chi = 1$ and $K_S^2 \geq -2$. 
\item[iii)] If  $X$ has non canonical singularities then  $S$ might not be minimal (e.g. $K_S^2 \leq 0$, even if $S$ of general type). This is the hardest problem to get hold on. Still we are struggling to get a complete solution (cf. section \ref{problems}).
\end{itemize}
\end{problem}

\section{The results}

We list here the results which were obtained in the last years in different collaborations with F. Catanese, F. Grunewald and R. Pignatelli.

\begin{theo}[\cite{bacat}, \cite{pg=q=0} \cite{4names},\cite{bp}]\label{classiso} \

1) Surfaces $S$ {\em isogenous to a product} (i.e., $S$ is an \'etale quotient of $C_1 \times C_2$ by a finite group $G$) with $p_g(S) = q(S) = 0$ form
17 irreducible connected components
of the moduli space $\mathfrak{M}_{(1,8)}^{can}$.

2) Surfaces with $p_g = 0$, whose canonical model is a singular
quotient  $X:=(C_1 \times C_2)/G$ by an
unmixed action of $G$ form 27 further irreducible families.

3) Minimal product-quotient surfaces with  $p_g=0$ such that the quotient model $X$ does not have
canonical singularities
form exactly further 32
irreducible families.

Moreover, $K^2_S = 8$ if and only if $S$ is isogenous to a product.

\end{theo}

The above results can be found in in tables \ref{K2>4} and \ref{K2<3}. 

$t_1, t_2$ are the signatures, $N$ denotes the number of irreducible families in the corresponding line. The other columns are self explanatory.

\begin{table}
\renewcommand{\arraystretch}{1,3}
\caption{Minimal product-quotient surfaces of general type with $p_g=0$, $K^2\geq
4$}\label{K2>4}
\tiny \begin{tabular}{|c|c|c|c|c|c|c|c|}
\hline
$K^2_S$&Sing X&$t_1$&$t_2$&$G$&N&$H_1(S,{\mathbb Z})$&$\pi_1(S)$\\
\hline\hline 8& $\emptyset$ & $2, 5^2$ &$3^4$ &  ${\mathfrak A}_5$
&$1$ &$\ZZ_3^2 \times
\ZZ_{15}$&$1 \rightarrow \Pi_{21} \times \Pi_{4} \rightarrow \pi_1
\rightarrow G \rightarrow 1$\\ 8&
$\emptyset$ & $5^3$ & $2^3,3$ & ${\mathfrak A}_5$ & $1$
&$\ZZ_{10}^2$&$1 \rightarrow \Pi_{6} \times
\Pi_{13} \rightarrow \pi_1 \rightarrow G \rightarrow 1$\\ 8&
$\emptyset$ & $3^2,5$ & $2^5$&
${\mathfrak A}_5$ & $1$ &$\ZZ_2^3 \times \ZZ_6$&$1 \rightarrow
\Pi_{16} \times \Pi_{5} \rightarrow
\pi_1 \rightarrow G \rightarrow 1$\\ 8& $\emptyset$ & $2,4,6$&  $2^6$
&  ${\mathfrak S}_4\times
\ZZ_2$ &$1$ &$\ZZ_2^4 \times \ZZ_4$&$1 \rightarrow \Pi_{25} \times
\Pi_{3} \rightarrow \pi_1
\rightarrow G \rightarrow 1$\\ 8& $\emptyset$ & $2^2, 4^2$ &$2^3, 4$
&  ${\rm G}(32,27)$  &$1$
&$\ZZ_2^2 \times \ZZ_4 \times \ZZ_8$&$1 \rightarrow \Pi_{5} \times
\Pi_{9} \rightarrow \pi_1
\rightarrow G \rightarrow 1$\\ 8& $\emptyset$ & $5^3$& $5^3$ &
$\ZZ_5^2$ &$2$ &$\ZZ_5^2$&$1
\rightarrow \Pi_{6} \times \Pi_{6} \rightarrow \pi_1 \rightarrow G
\rightarrow 1$\\ 8& $\emptyset$ &
$3,4^2$ & $2^6$ &  ${\mathfrak S}_4$ &$1$ &$\ZZ_2^4 \times \ZZ_8$&$1
\rightarrow \Pi_{13} \times
\Pi_{3} \rightarrow \pi_1 \rightarrow G \rightarrow 1$\\ 8&
$\emptyset$ & $2^2,4^2$&  $2^2,4^2$&
${\rm G}(16,3)$ &$1$ &$\ZZ_2^2 \times \ZZ_4 \times \ZZ_8$&$1
\rightarrow \Pi_{5} \times \Pi_{5}
\rightarrow \pi_1 \rightarrow G \rightarrow 1$\\ 8& $\emptyset$ &
$2^3,4$ & $2^6$ &  ${\rm
D}_4\times\ZZ_2$ &$1$ &$\ZZ_2^3 \times \ZZ_4^2$&$1 \rightarrow
\Pi_{9} \times \Pi_{3} \rightarrow
\pi_1 \rightarrow G \rightarrow 1$\\ 8& $\emptyset$ &$2^5$  & $2^5$&
$\ZZ_2^4$ &$1$
&$\ZZ_2^4$&$1 \rightarrow \Pi_{5} \times \Pi_{5} \rightarrow \pi_1
\rightarrow G \rightarrow 1$\\ 8&
$\emptyset$ & $3^4$ &  $3^4$&  $\ZZ_3^2$ & $1$ &$\ZZ_3^4$&$1
\rightarrow \Pi_{4} \times \Pi_{4}
\rightarrow \pi_1 \rightarrow G \rightarrow 1$\\ 8& $\emptyset$ &
$2^5$ &$2^6$ &  $\ZZ_2^3$ &$1$
&$\ZZ_2^6$&$1 \rightarrow \Pi_{3} \times \Pi_{5} \rightarrow \pi_1
\rightarrow G \rightarrow 1$\\
\hline\hline
   6&$1/2^2$&$2^3, 4$&$2^4, 4$& ${\mathbb Z}_2 \times D_4$ &1&
${\mathbb Z}_2^2 \times {\mathbb
Z}_4^2$ & $1\rightarrow {\mathbb Z}^2 \times \Pi_2 \rightarrow \pi_1
\rightarrow {\mathbb Z}_2^2
\rightarrow 1$\\
   6&$1/2^2$&$2^4,4$ & $2, 4, 6$ &${\mathbb Z}_2 \times {\mathfrak
S}_4$ &1&   ${\mathbb Z}_2^3
\times {\mathbb Z}_4$ &  $1\rightarrow \Pi_2 \rightarrow \pi_1
\rightarrow  {\mathbb Z}_2
\times{\mathbb Z}_4 \rightarrow 1$   \\
   6&$1/2^2$&$2, 5^2$ & $2, 3^3$ & ${\mathfrak A}_5$&1&${\mathbb Z}_3
\times {\mathbb
Z}_{15}$&${\mathbb Z}^2 \rtimes {\mathbb Z}_{15}$\\
   6&$1/2^2$&$2, 4, 10$&$2, 4, 6$&${\mathbb Z}_2 \times {\mathfrak
S}_5$&1&${\mathbb Z}_2 \times
{\mathbb Z}_4$&${\mathfrak S}_3 \times D_{4,5,-1}$\\
   6&$1/2^2$&$2, 7^2$&$3^2, 4$&PSL(2,7)&2&${\mathbb Z}_{21}$ &
${\mathbb Z}_7 \times {\mathfrak
A}_4$\\
   6&$1/2^2$&$2, 5^2$&$3^2, 4$&${\mathfrak A}_6$&2&${\mathbb
Z}_{15}$&${\mathbb Z}_5 \times
{\mathfrak A}_4$\\
\hline\hline
   5&$1/3, 2/3$&$2, 4, 6$&$2^4, 3$&${\mathbb Z}_2 \times {\mathfrak
S}_4$&1&${\mathbb Z}_2^2 \times
{\mathbb Z}_4$&$ 1\rightarrow {\mathbb Z}^2 \rightarrow \pi_1
\rightarrow  D_{2,8,3} \rightarrow 1$\\
   5&$1/3, 2/3$&$2^4, 3$&$3, 4^2$&${\mathfrak S}_4$ & 1 &   ${\mathbb
Z}_2^2 \times {\mathbb
Z}_8$&$ 1 \rightarrow {\mathbb Z}^2 \rightarrow \pi_1 \rightarrow
{\mathbb Z}_8 \rightarrow 1$\\
   5&$1/3, 2/3$&$4^2, 6$&$2^3, 3$&${\mathbb Z}_2 \times {\mathfrak
S}_4$ &1& ${\mathbb Z}_2^2 \times
{\mathbb Z}_8$& $1 \rightarrow {\mathbb Z}^2 \rightarrow \pi_1
\rightarrow {\mathbb Z}_8 \rightarrow
1$\\
   5&$1/3, 2/3$ & $2, 5, 6$ &$3, 4^2$&${\mathfrak S}_5$ &1&${\mathbb
Z}_8$ &              $
D_{8,5,-1}$              \\
   5&$1/3, 2/3$ & $3, 5^2$ &$2^3, 3$&${\mathfrak A}_5$ &1&${\mathbb
Z}_2 \times {\mathbb Z}_{10}$ &
${\mathbb Z}_5 \times Q_8$               \\
   5&$1/3, 2/3$&$2^3, 3$ &$3, 4^2$ &${\mathbb Z}_2^4 \rtimes {\mathfrak
S}_3$ &1& ${\mathbb Z}_2
\times {\mathbb Z}_8$ & $D_{8,4,3}$?               \\
   5&$1/3, 2/3$&$3, 5^2$ &$2^3, 3$ &${\mathfrak A}_5$ &1& ${\mathbb
Z}_2 \times {\mathbb Z}_{10}$ &
${\mathbb Z}_2 \times {\mathbb Z}_{10}$              \\
\hline\hline
   4&$1/2^4$&$2^5$&$2^5$&${\mathbb Z}_2^3$ &1& ${\mathbb Z}_2^3 \times
{\mathbb Z}_4$ &
$1\rightarrow {\mathbb Z}^4 \rightarrow \pi_1 \rightarrow {\mathbb
Z}_2^2 \rightarrow 1$       \\
   4&$1/2^4$&$2^2, 4^2$ &$2^2, 4^2$&${\mathbb Z}_2 \times {\mathbb
Z}_4$ &1& ${\mathbb Z}_2^3
\times {\mathbb Z}_4$ &   $1\rightarrow {\mathbb Z}^4 \rightarrow
\pi_1 \rightarrow {\mathbb Z}_2^2
\rightarrow 1$       \\
   4&$1/2^4$&$2^5$&$2^3, 4$&${\mathbb Z}_2 \times D_4$ &1& ${\mathbb
Z}_2^2 \times {\mathbb
Z}_4$ & $1\rightarrow {\mathbb Z}^2 \rightarrow \pi_1 \rightarrow
{\mathbb Z}_2 \times {\mathbb Z}_4
\rightarrow 1$ \\
   4&$1/2^4$&$3, 6^2$&$2^2, 3^2$&${\mathbb Z}_3 \times {\mathfrak S}_3$
&1&${\mathbb Z}_3^2$
&${\mathbb Z}^2 \rtimes {\mathbb Z}_3$           \\
   4&$1/2^4$&$3, 6^2$&$2, 4, 5$&${\mathfrak S}_5$ &1&${\mathbb Z}_3^2$
&${\mathbb Z}^2 \rtimes
{\mathbb Z}_3$          \\
   4&$1/2^4$&$2^5$&$2, 4, 6$&${{\mathbb Z}_2 \times \mathfrak S}_4$
   &1&${\mathbb Z}_2^3$ & ${\mathbb Z}^2 \rtimes {\mathbb Z}_2$ \\
   4&$1/2^4$&$2^2, 4^2$&$2, 4, 6$&${\mathbb Z}_2 \times {\mathfrak
S}_4$ &1&${\mathbb Z}_2^2 \times
{\mathbb Z}_4$ & ${\mathbb Z}^2 \rtimes {\mathbb Z}_4$\\
   4&$1/2^4$&$2^5$&$3, 4^2$&${\mathfrak S}_4$ &1& ${\mathbb Z}_2^2
\times {\mathbb Z}_4$ &
${\mathbb Z}^2 \rtimes {\mathbb Z}_4$\\
   4&$1/2^4$&$2^3, 4$&$2^3, 4$&${\mathbb Z}_2^4 \rtimes {\mathbb Z}_2$
&1& ${\mathbb Z}_4^2$ &
$G(32, 2)$\\
   4&$1/2^4$&$2, 5^2$&$2^2, 3^2$&${\mathfrak A}_5$ &1& ${\mathbb
Z}_{15}$ & ${\mathbb
Z}_{15}$                  \\
   4&$1/2^4$&$2^2, 3^2$&$2^2, 3^2$& ${\mathbb Z}_3^2 \rtimes Z_2$ &1&${\mathbb
Z}_3^3$ &
${\mathbb Z}_3^3$                  \\
   4&$2/5^2$&$2^3, 5$&$3^2, 5$&${\mathfrak A}_5$ &1&    ${\mathbb Z}_2
\times {\mathbb Z}_6$  &
${\mathbb Z}_2 \times {\mathbb Z}_6$  \\
   4&$2/5^2$&$2, 4, 5$&$4^2, 5$& ${\mathbb Z}_2^4 \rtimes D_5$ &3&
${\mathbb Z}_8$  &
${\mathbb Z}_8$?                   \\
   4&$2/5^2$&$2, 4, 5$&$3^2, 5$& ${\mathfrak A}_6$ &1&
${\mathbb Z}_6$  &  ${\mathbb
Z}_6$                   \\
\hline
\end{tabular}
\end{table}

\begin{table}
\caption{Minimal product-quotient surfaces of general type with $p_g=0$, $K^2\leq 3$}
\label{K2<3}
\renewcommand{\arraystretch}{1,3}
\tiny \begin{tabular}{|c|c|c|c|c|c|c|c|}
\hline
$K^2_S$&Sing X&$t_1$&$t_2$&$G$&N&$H_1(S,{\mathbb Z})$&$\pi_1(S)$\\
\hline\hline
   3&$1/5, 4/5$ &$2^3, 5$&$3^2, 5$& ${\mathfrak A}_5$ &1&    ${\mathbb
Z}_2 \times {\mathbb Z}_6$
&              $ {\mathbb Z}_2 \times {\mathbb Z}_6$               \\
   3&$1/5, 4/5$ &$2, 4, 5$&$4^2, 5$& ${\mathbb Z}_2^4 \rtimes D_5$ &3&
${\mathbb Z}_8$
&                $ {\mathbb Z}_8 $?                 \\
   3&$1/3, 1/2^2, 2/3$ & $2^2, 3, 4$ &$2, 4, 6$&    ${\mathbb Z}_2
   \times {\mathfrak S}_4$ &1&    ${\mathbb Z}_2 \times {\mathbb Z}_4$
&              $ {\mathbb Z}_2 \times
{\mathbb Z}_4$               \\
   3&$1/5, 4/5$&$2, 4, 5$&$3^2, 5$&    ${\mathfrak A}_6$ &1&
${\mathbb Z}_6$  &
${\mathbb Z}_6$                   \\
\hline\hline
   2&$1/3^2, 2/3^2$&$2, 6^2$ &$2^2, 3^2$& ${\mathbb Z}_2 \times
{\mathfrak A}_4$ &1& ${\mathbb
Z}_2^2$ & $Q_8$                   \\
   2&$1/2^6$&$4^3$ &$4^3$ &${\mathbb Z}_4^2$ &1&${\mathbb Z}_2^3$
&${\mathbb Z}_2^3$
\\
   2&$1/2^6$&$2^3, 4$ &$2^3, 4$ &${\mathbb Z}_2 \times D_4$ &1&
${\mathbb Z}_2 \times {\mathbb
Z}_4$ & $  {\mathbb Z}_2 \times {\mathbb Z}_4$                \\
   2&$1/3^2, 2/3^2$&$2^2, 3^2$&$3, 4^2$&${\mathfrak S}_4$ &1& ${\mathbb
Z}_8$ & ${\mathbb
Z}_8$                   \\
   2&$1/3^2, 2/3^2$&$3^2, 5$ &$3^2, 5$ &${\mathbb Z}_5^2 \rtimes
{\mathbb Z}_3$ &2&         $
{\mathbb Z}_5$ &   $              {\mathbb Z}_5$?                 \\
   2&$1/2^6$&$2, 5^2$&$2^3, 3$&${\mathfrak A}_5$ &1&${\mathbb Z}_5$
&${\mathbb Z}_5$
\\
   2&$1/2^6$&$2^3, 4$&$2, 4, 6$&${\mathbb Z}_2 \times {\mathfrak
S}_4$&1&$         {\mathbb Z}_2^2$
&$                {\mathbb Z}_2^2$                  \\
   2&$1/3^2, 2/3^2$&$3^2, 5$ &$2^3, 3$ &${\mathfrak A}_5$ &1&${\mathbb
Z}_2^2$ & ${\mathbb
Z}_2^2$                  \\
   2&$1/2^6$&$2, 3, 7$ &$4^3$ & PSL(2,7) &2& ${\mathbb Z}_2^2$
&${\mathbb Z}_2^2$                  \\
   2&$1/2^6$&$2, 6^2$&$2^3, 3$&${\mathfrak S}_3 \times {\mathfrak S}_3$
&1&$           {\mathbb Z}_3$
&$                 {\mathbb Z}_3 $                  \\
   2&$1/2^6$&$2, 6^2$&$2, 4, 5$&${\mathfrak S}_5$&1&${\mathbb
Z}_3$&${\mathbb Z}_3$                  \\
   2&$1/4, 1/2^2, 3/4$&$2, 4, 7$&$3^2, 4$&     PSL(2,7) &2& ${\mathbb
Z}_3$ &   $              {\mathbb Z}_3
$                  \\
   2&$1/4, 1/2^2, 3/4$&$2, 4, 5$&$3^2, 4$&     $     {\mathfrak
A}_6$&2&  $ {\mathbb Z}_3$ &      $
{\mathbb Z}_3$                   \\
   2&$1/4, 1/2^2, 3/4$&$2, 4, 5$&$3, 4, 6$&   ${\mathfrak S}_5$ &2&
$ {\mathbb Z}_3$ &
${\mathbb Z}_3 $                  \\
\hline\hline
   1&$1/3, 1/2^4, 2/3$&$2^3, 3$ &$3, 4^2$&${\mathfrak S}_4$
&1&${\mathbb Z}_4$ & ${\mathbb
Z}_4$                   \\
   1&$1/3, 1/2^4, 2/3$&$2, 3, 7$&$3, 4^2$&     PSL(2,7) &1& ${\mathbb
Z}_2$ & ${\mathbb
Z}_2$                   \\
   1&$1/3, 1/2^4, 2/3$&$2, 4, 6$&$2^3, 3$&${\mathbb Z}_2 \times
{\mathfrak S}_4$ &1&   ${\mathbb
Z}_2$ & ${\mathbb Z}_2$ \\
\hline
\end{tabular}
\end{table}

Comparing tables \ref{K2>4} and \ref{K2<3} with the list in table 1 of 
\cite{surveypg=0}, we note

\begin{cor}
Minimal surfaces of general type with $p_g = q = 0$ and with  $3 \leq K^2 \leq 6$ realize at 
least 45 topological types.
\end{cor}

Note that before proving the results summarized in theorem \ref{classiso}
only $12$ topological types of surfaces of general type with $p_g = q= 0$ and with $3 \leq K^2 \leq 6$ 
were known. 
Surfaces with $p_g = 0$ are also very interesting in view of Bloch's conjecture (\cite{bloch}),
claiming that for surfaces with $p_g  =  0$
the group of zero cycles modulo rational equivalence
is  isomorphic to $\ZZ$.

Using Kimura's results (\cite{kimura}, see also \cite{gp}), the present results,
and those of the previous papers \cite{bacat}, \cite{pg=q=0}, \cite{4names}, we get for the first time a substantial amount of surfaces confirming Bloch's conjecture.
\begin{cor}
All the families in theorem \ref{classiso} fulfill Bloch's conjecture, i.e., there are 77 families  of 
surfaces  of general type with
$p_g = 0$ for which Bloch's conjecture holds.
\end{cor}

There remains open the following:

\begin{problem}
 Classify all product-quotient surfaces of general type with $p_g=0$.
\end{problem}
In view of theorem \ref{classiso} it remains to classify all {\em non-minimal} product-quotient surfaces of general type.

\section{Non-minimal product-quotient surfaces: finding rational curves}

In order to finish the classification of product-quotient surfaces $S$ of general type with $p_g = 0$ we need:
\begin{itemize}
 \item to find an integer $C$ such that $K_S^2 \leq C$ implies that $S$ is not of general type;
\item if we have given a product-quotient surface $S$, either
\begin{itemize}
\item show that $S$ cannot be of general type,
\item prove that $S$ is minimal, or
\item find the exceptional curves of the first kind on $S$.
\end{itemize}
\end{itemize}

\begin{rem}
 Note that rational curves on $S$ can appear as 
\begin{itemize}
 \item components of singular fibers, or
\item they have to pass through the singular points at least three times (counted with multiplicities).
\end{itemize}

\end{rem}

We need to consider the following diagram
\begin{equation}\label{diagram1}
\xymatrix{
&C_1 \times C_2 \ar^{p_2}[dr]\ar_{p_1}[dl]\ar_{\lambda_{12}}[dd]&\\
C_1\ar_{\lambda_1}[dd]&&C_2\ar^{\lambda_2}[dd]\\
&X = (C_1 \times C_2)/G\ar_{f_1}[ld]\ar^{f_2}[dr]\ar^{\lambda}[dd]&\\
C_1/G \cong \PP^1&&C_2/G \cong \PP^1\\
&C_1/G \times C_2/G \cong \PP^1 \times \PP^1 \ar_{}[ul]\ar_{}[ur]&\\
}
\end{equation}

\medskip

Assume that $\Gamma \subset X$ is a (possibly singular) rational curve. Let $\bar{\Gamma} := \lambda_{12}^*(\Gamma) = \sum_{1}^k n_i \Gamma_i$ be the decomposition in irreducible components of its pull back to $C_1 \times C_2$.

Observe that $n_i=1, \ \forall i$ (since $\lambda_{12}$ has discrete ramification), and that $G$ acts transitively on the set $\{ \Gamma_i | i \in \{1, \ldots ,k \} \}$. Hence there is a subgroup $H \leq G$ of index $k$ acting on $\Gamma_1$ such that $\lambda_{12}(\Gamma_1) = \Gamma_1 /H  = \Gamma$.

Normalizing $\Gamma_1$ and $\Gamma$, we get the following commutative diagram:
\begin{equation}\label{normalization}
\xymatrix{
\tilde{\Gamma}_1\ar[r]\ar_{\gamma}[d]&\Gamma_1\ar[d]\\
\PP^1\ar^{\nu}[r]&\Gamma\\
}
\end{equation}
and, since each automorphism lifts to the normalization, $H$ acts on $\tilde{\Gamma}_1$ and $\gamma$ is the quotient map $\tilde{\Gamma}_1 \rightarrow \tilde{\Gamma}_1 /H \cong \PP^1$.

\begin{lemma}\label{branchpoints}
Let $p$ be a branch point of $\gamma$ of multiplicity $m$. Then $\nu (p)$ is a singular point of $X$ of
type $\frac{1}{n} (1,a)$, where $m | n$.
\end{lemma}

Using elementary properties of the intersection form on surfaces of general type, we get the following:

\begin{prop}\label{countingpoints}
Let $S$ be a product-quotient surface of general type. Let $\pi \colon S \rightarrow X$ be the minimal resolution of singularities of the quotient model. Assume that  $\pi_*^{-1}(\Gamma)$
is a $(-1)$-curve in $S$ and let  $x \in \Sing (X)$  be a point of type $\frac{1}{n} (1,a)$, with $\frac{n}{a}=[b_1, \ldots , b_r]$. Consider the map $\nu$ in diagram (\ref{normalization}).
Then
\begin{itemize}
 \item[i)] $ \# \nu^{-1}(x) \leq 1$, if $a = n-1$;
\item[ii)] $ \# \nu^{-1}(x) \leq \sum_{\{b_i \geq 4\}} (b_i-3) + \# \{i : b_i = 3\}$, if $a \neq n-1$.
\end{itemize}
\end{prop}

This is sufficient to show the minimality of all product-quotient surfaces with $p_g=0$ and $K_S^2 \geq 1$ with the exception of one case, which we call the {\em fake Godeaux surface}, cf. section \ref{fakegodeaux}. For a detailed account of these arguments we refer to the original paper \cite{bp}.

\begin{theo}\label{minimal}
The minimal product-quotient surfaces of general type with $p_g=0$ form $72$ families which are
listed in tables \ref{K2>4} and \ref{K2<3}.
\end{theo}

\subsection{The fake Godeaux surface}\label{fakegodeaux}
Our MAGMA code, which can be downloaded from 
\begin{verbatim}
http://www.science.unitn.it/~pignatel/papers.html.
\end{verbatim} produces $73$ families of product-quotient surfaces
of general type with $p_g=0$ and $K^2>0$, and theorem \ref{minimal} shows that $72$ of them are families of minimal surfaces.

The 73rd output is exactly one pair of appropriate orbifold homomorphisms, which we will describe in the sequel.

We see $G=PSL(2,7)$ as subgroup of ${\mathfrak S}_8$
generated by $(367)(458),(182)(456)$. Then 
%\begin{align}
\begin{flalign*}
\varphi_1 \colon \BT(7,3,3) & \ra G,&\varphi_2 \colon \BT(7,4,2)& \ra G\\
c_1& \mapsto  (1824375)&c_1 & \mapsto (1658327)\\
c_2& \mapsto  (136)(284)&c_2 & \mapsto (1478)(2653)\\
c_3& \mapsto  (164)(357)&c_3 & \mapsto (15)(23)(36)(47).
\end{flalign*}

The pair $(\varphi_1, \varphi_2)$ above determines exactly one product-quotient surface $S$, which we have called "the fake Godeaux surface". Its topological fundamental group is the cyclic group of order six. 

The notation is explained by the fact that minimal surfaces of general type with $K_S^2 = 1$, $p_g=0$ are called Godeaux surfaces.

By a result of M. Reid (cf. \cite{tokyo}) the order of the algebraic fundamental group of a Godeaux surface is at most five, implying that our surface has to be non-minimal.

Note that $S$ is a surface of general type.

In fact, in \cite{bp} it is shown:
\begin{theo}
The fake Godeaux surface $S$ has two $(-1)$-curves. Its minimal model has $K^2=3$.
\end{theo}

We briefly recall the construction of one of the $(-1)-$curves on $S$, for details we refer to \cite{bp}.

We can choose the branch points $p_i$ of $\lambda_1 \colon C_1 \rightarrow \PP^1$ and $p'_j$ of $\lambda_2$ at our convenience. In fact, assume $(p_1, p_2, p_3)=(1,0,\infty)$, $(p'_1, p'_2, p'_3)=(0,\infty,-\frac{9}{16})$.

Consider the normalization $\hat{C}'_1$ of the fibre product between
$\lambda_1$ and the $\ZZ/3\ZZ$-cover $\xi' \colon \PP^1 \rightarrow \PP^1$
defined by $\xi'(t)=t^3$. We have a diagram
\begin{equation*}
\xymatrix{
\hat{C}'_1\ar^{\hat{\xi}'}[r]\ar^{\hat{\lambda}'_1}[d]&C_1\ar^{\lambda_1}[d]\\
\PP^1\ar^{\xi'}[r]&\PP^1 \\
}
\end{equation*}
where the horizontal maps are $\ZZ/3\ZZ$-covers and the vertical maps
are $PSL(2,7)$-covers. 

The branch points of $\hat{\lambda}'_1$ are the three points in $\xi'^{-1}(p_1)$, all with branching index 7.

For $C_2$, we take the normalized fibre product between $\lambda_2$ and the map $\eta' \colon
\PP^1 \rightarrow \PP^1$ defined by $\eta'(t)=\frac{(t^3-1)(t-1)}{(t+1)^4}$.

Note that $\eta'$ has degree $4$ and factors
through the involution $t \mapsto \frac1t$. Therefore it is the
composition of two double covers, say $\eta'=\eta_1' \circ \eta_2'$.
We get the following diagram

\begin{equation*}
\xymatrix{
\hat{C}'_2\ar^{\hat{\eta}'}@/^1.7pc/[rr]\ar^{\hat{\eta}'_2}[r]\ar^{\hat{\lambda}'_2}[d]&\bar{C}_2'\ar[d]^{\bar{\lambda}_2'}\ar^{\hat{\eta}'_1}[r]&C_2\ar^{\lambda_2}[d]\\
\PP^1\ar_{\eta'}@/_1pc/[rr]\ar^{\eta'_2}[r]&\PP^1\ar^{\eta'_1}[r]&\PP^1\\
}
\end{equation*}
where the horizontal maps are $\ZZ/2\ZZ$-covers and the vertical maps
are $PSL(2,7)$-covers.

Then the branch points of $\hat{\lambda}'_2$ are the three points of $(\eta')^{-1}(p'_1)$, each with branching index 7.

\begin{lemma}\label{conjugated}
$(\hat{C}'_1,\hat{\lambda}'_1)$ and $(\hat{C}'_2,\hat{\lambda}'_2)$ are isomorphic as Galois covers of $\PP^1$.
\end{lemma}
Consider the curve $\hat{C}':=\hat{C}'_1=\hat{C}'_2$ and let $C':=(\hat{\xi}',\hat{\eta}')(\hat{C}')\subset C_1
\times C_2$. $C'$  is $G-$invariant, and the quotient is a
rational curve $\hat{C}'/G \cong \PP^1 \stackrel{e'}{\rightarrow} D'$ contained in the quotient model $X$ of the fake Godeaux surface $S$.

\begin{prop}\label{E'}
Let $E'$ be the strict transform of
$D'$ on $S$. Then $E'$ is a smooth rational curve
with self intersection $-1$.
\end{prop}

In a similar way we find a second $(-1)$-curve $E''$ on $S$. Looking at the dual graph (cf. diagram (\ref{configS'})) of the configuration of the rational curves $E'$, $E''$ and the Hirzebruch-Jung strings over the singular points of $X$, it follows easily that blowing down these two exceptional curves we get the minimal model $S'$ of $S$ with invariants $K_{S'}^2 =3$, $p_g=0$.

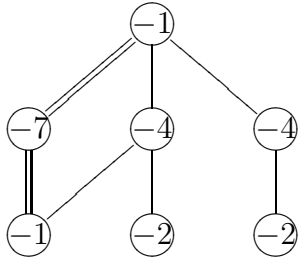
\begin{equation}\label{configS'}
\xymatrix{
&*+[o][F]{-1}\ar@{-}[d]\ar@{=}[dl]\ar@{-}[dr]&\\
*+[o][F]{-7}\ar@{=}[d]&*+[o][F]{-4}\ar@{-}[d]\ar@{-}[dl]&*+[o][F]{-4}\ar@{-}[d]\\
*+[o][F]{-1}&*+[o][F]{-2}&*+[o][F]{-2}\\
}
\end{equation}

\begin{rem}
The fake Godeaux surface and another product-quotient surface $S$ with $K_S^2 = -1$ are up to now the only known non-minimal product-quotient surfaces of general type.
\end{rem}

\section{Problems and new input} \label{problems}
\subsection{Hodge theoretic information of product-quotient surfaces} \

\noindent
We start with the following:
\begin{prop}\label{parity}
Let $X:=(C_1 \times C_2)/G$ be the quotient model of a product-quotient surface. Then 
\begin{itemize}
 \item $\dim H^2(X) \equiv 0 \mod 2$,
\item $\dim H^2(X) \geq 2$.
\end{itemize}

\end{prop}
\begin{proof}
By the Hodge decomposition and K\"unneth's formula we get: 

\begin{align*}
H^2(X) \cong & H^2({C_1 \times C_2},\CC)^G \\
\cong & H^0(\Omega^2_{C_1 \times C_2})^G \oplus H^1(\Omega^1_{C_1 \times C_2})^G
\oplus (H^0(\Omega^2_{C_1 \times C_2})^G)^*.
\end{align*}

Therefore $h^2(X,\CC) = 2 \cdot h^0(\Omega^2_{C_1\times C_2})^G + h^1(\Omega^1_{C_1\times C_2})^G$, whence the claim is proven once we show that $h^1(\Omega^1_{C_1\times C_2})^G \equiv 0 \mod 2$.

We recall that 
by K\"unneth's formula (cf. e.g. \cite{kaup}) and Hodge theory

\begin{align*}
H^1(\Omega^1_{C_1 \times C_2})^G\cong& \left( H^1(\Omega^1_{C_1}) \otimes H^0(\hol_{C_2}) \right)^G \oplus
\left( H^1(\Omega^1_{C_2}) \otimes H^0(\hol_{C_1}) \right)^G\\
\oplus& \left( H^0(\Omega^1_{C_1}) \otimes \overline{H^0(\Omega^1_{C_2})} \right)^G \oplus \left( \overline{H^0(\Omega^1_{C_1})} \otimes H^0(\Omega^1_{C_2}) \right)^G\\
\cong & \ \CC^2 \oplus V \oplus \bar{V},
\end{align*}
where 
$$V:= \left( H^0(\Omega^1_{C_1}) \otimes \overline{H^0(\Omega^1_{C_2})} \right)^G.$$

\end{proof}

Let $\sigma \colon S \rightarrow X$ be the minimal resolution of the singularities of $X$. Then we can show the following:

\begin{prop}\label{hodge} \
 \begin{enumerate}
  \item Let $S$ be a product-quotient surface with $q(S) = 0$. Then $H^2(S, \CC) \cong H^2 (X, \CC) \oplus \CC^l$, where $l$ is the number of irreducible components of the exceptional locus of $\sigma$ (which consists of Hirzebruch-Jung strings over each singular point of $X$).
\item If $p_g(S) = 0$, then $H^0(C_1 \times C_2, \Omega^2_{C_1 \times C_2})^G = 0$. In particular, $H^2(X, \CC) \cong H^1(C_1 \times C_2, \Omega^1_{C_1 \times C_2})^G$.
 \end{enumerate}

\end{prop}

\

\begin{rem} \
\begin{enumerate}
 \item Proposition \ref{hodge} shows that the singularities of the quotient-model $X$ give no conditions of adjunction for canonical forms, even if the singularities are not canonical. This is not true for the bicanonical divisor.
\item The above results make clear that the condition that $S$ has vanishing geometric genus gives strong restrictions on the $G$-modules $H^0(C_i, \Omega^1_{C_i})$. For example, we can list the following properties:
\begin{enumerate}
 \item if $\chi$ is an irreducible real character of $G$, then $H^0(\Omega^1_{C_1})^{\chi} = 0$ or $H^0(\Omega^1_{C_2})^{\bar{\chi}} = 0$;
\item $\dim H^2(X, \CC) > 2$ if and only if there is an irreducible non real character $\chi$ of $G$ such that $H^0(\Omega^1_{C_1})^{\chi} \neq 0$ and $H^0(\Omega^1_{C_2})^{\chi} \neq 0$.
\end{enumerate}
Each time that such a situation occurs, it raises the dimension of $\dim H^2(X, \CC)$ by two.

\end{enumerate}

\end{rem}

The last statement can be shown  calculating 
$$
V:= \left( H^0(\Omega^1_{C_1}) \otimes \overline{H^0(\Omega^1_{C_2})} \right)^G.
$$
with the following version of Schur's lemma (cf. e.g. \cite{langrep}):

\begin{lemma}
 Let $G$ be a finite group and let $W$ be an irreducible $G$-representation. Then
\begin{enumerate}
 \item $\dim (W \otimes W^*)^G = 1$;
\item if $W'$ is an irreducible $G$-representation not isomorphic to $W^*$, then $\dim (W \otimes W')^G = 0$.
\end{enumerate}
\end{lemma}

\begin{problem}
 Determine $\dim H^2(X, \CC) -2$. Or at least give a reasonable upper bound.
\end{problem}

\begin{cor}
 Let $G$ be a group, where all irreducible characters are real (i.e. self dual). Then for each product-quotient surface $S$ with $p_g=q=0$ and group $G$ we have $\dim H^2(X, \CC) = 2$.
\end{cor}

\subsection{New numerical invariants of product-quotient surfaces} \

\noindent
Since the algorithm developed in \cite{bp} is getting very time and memory consuming when $K^2$ is getting small. To be more precise, it seems very unlikely that it is possible to push the calculations to values of $K^2 \leq -3$. 
This is a serious obstacle for finishing the classification of product-quotient surfaces of general type with $\chi(S) =1$ even if a lower bound for $K^2$ can be found.

Therefore it seems natural to look for different invariants of product-quotient surfaces.

We introduce the following invariants of a cyclic quotient singularity of type $\frac{1}{n}(1,q)$ (recall that $1 \leq q \leq n-1$, $\gcd(n,q) =1$ and $q'$ is the multiplicative inverse of $q\mod n$).

\

\begin{defin} \
\begin{itemize}
\item[i)] the continued fraction
$$
\frac{n}{q} = b_1 - \frac{1}{b_2 - \frac{1}{b_3 - \ldots}} =:[b_1, \ldots, b_l];
$$
 $\frac{q}{n}=[b_1,\ldots,b_l]$, $b_i \in \NN$, $b_i \geq 2$;
\item[ii)] $l\left(\frac{1}{n}(1,q)\right)$ is the length of the continued fraction;
\item[iii)] $\gamma\left(\frac{1}{n}(1,q)\right):=\frac16 \left[ \frac{q+q'}n + \sum_{i=1}^l (b_i-3)\right]$;
\item[iv)] $\mu\left(\frac{1}{n}(1,q)\right)=1-\frac1n$.
\item[v)] the index of the singularity $I\left(\frac{1}{n}(1,q)\right)=\frac{n}{\gcd(n,q+1)}$.
\end{itemize}
\end{defin}

It is immediate how to globalize these invariants and define the same invariants also for the basket of singularities of the quotient model of a product-quotient surface. 

\begin{defin}[\bf Invariants of the basket $\mathfrak{B}$]
Let ${\mathfrak B}$ be the basket of singularities of a product-quotient surface (note that the only hypothesis we need is that all singularities in the basket are cyclic quotient singularities):
\begin{itemize}
\item[] $l:=l(\mathfrak{B}):=\sum_{x \in {\mathfrak B}} l(x)$;
\item[] $\gamma:= \gamma(\mathfrak{B}):=\sum_{x \in {\mathfrak B}} \gamma(x)$;
\item[] $\mu:= \mu(\mathfrak{B}) :=\sum_{x \in {\mathfrak B}} \mu(x)$;
\item[] $I:= I(\mathfrak{B}) :=\lcm_{x \in {\mathfrak B}} I(x)$.
\end{itemize}
\end{defin}

\begin{rem}
The invariants $e,k, B$ introduced in \cite{bp}) (cf. definition \ref{oldinv}) depend on $l, \mu, \gamma$ as follows:

\begin{itemize}
\item[-] $e\left(\frac{1}{n}(1,q)\right)=l\left(\frac{1}{n}(1,q)\right)+\mu\left(\frac{1}{n}(1,q)\right)$;
\item[-] $k\left(\frac{1}{n}(1,q)\right)=6\gamma\left(\frac{1}{n}(1,q)\right)+l\left(\frac{1}{n}(1,q)\right)-2\mu\left(\frac{1}{n}(1,q)\right)$;
\item[-] $B\left(\frac{1}{n}(1,q)\right)=3\left(2\gamma\left(\frac{1}{n}(1,q)\right)+l\left(\frac{1}{n}(1,q)\right)\right)$.
\end{itemize}
The same relations as above also hold for the global invariants $e,k,B$ and $l, \mu, \gamma$.
\end{rem}

Let $\sigma \colon S \rightarrow X$ be the minimal resolution of singularities (of the quotient model). Then the invariants $K^2_S$, $\chi:=\chi(\hol_S)$ are related to the basket 
${\mathfrak B} (= \mathfrak{B}(X))$ (in terms of the above defined invariants) as follows:

\begin{prop}[\cite{bp}, prop. 2.8] \

\begin{itemize}
\item $K^2_S=8\chi(\hol_S)-2\gamma-l$;
\item $\chi(\hol_S)=\frac{(g_1-1)(g_2-1)}{|G|}+\frac{\mu-2\gamma}{4}$.
\end{itemize}
\end{prop}

An interesting feature is that, for product quotient surfaces $S$ with $p_g=0$, the invariant $\gamma$ defined above is exactly $\frac 12 (h^{1,1}(X)-2) \in \NN$ (cf. proposition \ref{parity}, \ref{hodge}).

\begin{prop}\label{nat}
Let $S$ be a product-quotient surface. Then 
$$
\gamma+p_g(S) \in \NN.
$$
\end{prop}
\begin{proof}
It is easily seen (from the intersection form on $H^2(S, \CC)$) that the classes of the fibers of the two fibrations $S \rightarrow C_i/G$, and the $l$ classes of the irreducible exceptional curves of $\sigma$ are linearly independent in $H^1(S, \Omega^1_S)$. Therefore we have 
$$
h^{1,1}-l-2\in \NN.
$$
By proposition \ref{hodge}, 1), we know that $\dim H^2(S, \CC)=l+\dim H^2(X, \CC)$ and, by lemma \ref{parity}, we see that $h^{1,1}$ has the same parity as $l$. Therefore $h^{1,1}-l-2\in 2\NN$.

The claim follows now from the following:
\begin{align*}
2(\gamma+p_g)=&-K^2_S+8\chi-l+2p_g\\
=&c_2(S)-4\chi-l+2p_g\\
=&2-2b_1+b_2-4+4q-4p_g-l+2p_g\\
=&h^{1,1}-l-2.\\
\end{align*}
\end{proof}

In fact, it can be shown that the classification problem of product quotient surfaces with fixed $\chi$ and $\gamma \in \NN$ is finite. 

In fact, in \cite{bp2}, the following result is shown:

\begin{theo}
 Let $\gamma, \chi \in \NN$ be arbitrary, but fixed. Then there is 
\begin{itemize}
 \item a finite number of possible signatures;
\item a finite number of possible baskets
\end{itemize}
for product quotient surfaces with $\chi(S) = \chi$ and $\gamma(S) = \gamma$.
\end{theo}

\begin{rem}
In order to implement an algorithm, we of course need explicit bounds not an abstract existence result. In \cite{bp2} we give all the bounds explicitly as a function of $\gamma$ and $\chi$. Moreover, we have further restrictions on the signatures and baskets which were already established in \cite{bp}. Even if we did not yet completely implement the algorithms, we are quite optimistic that these new invariants will make the computations more feasible.
\end{rem}

Our aim is to find an algorithm (and implement it!) which calculates {\em all} regular product-quotient surfaces {\em of general type} with fixed $\chi(S) = 1$. 

The results of \cite{bp2} explained above show that we have an algorithm if we additionally fix $\gamma \in \NN$. 

For different reasons it is more useful to work with the (equivalent) invariant  $c:=\gamma+p_g \in \NN$, which is the half of the codimension of the subspace of $H^2(S,\CC)$ generated by the classes of the irreducible components of the Hirzebruch-Jung strings over the singular points of $X$ and the classes of the two fibers (cf. proof of proposition \ref{nat}).

\begin{rem}
All surfaces in tables \ref{K2>4}, \ref{K2<3} have $\gamma=0$. Only the fake Godeaux surface has non-vanishing $\gamma$, namely $\gamma = 1$. This lead us to conjecture that the appearance of exceptional curves of the first kind and non-vanishing $\gamma$ have to do with each other.
\end{rem}

We have the following:
\begin{conj}\label{boundc}
 There is an explicit function $C = C(p_g,q)$ such that 
$$
c \leq C(p_g,q).
$$
\end{conj}

\begin{definition}
Let $S$ be a regular product-quotient surface with quotient model $X$. Denote as usual the minimal resolution of singularities $S \rightarrow X$ by $\sigma$.

By proposition \ref{hodge}, we have:
$$
H^2(S)=H^2(X) \oplus L,
$$
where $L$ is the $l$-dimensional subspace of $H^2(S)$ generated by the classes of the irreducible curves of the exceptional locus of $\sigma$.

Furthermore, we consider the Zariski decomposition of $K_S$:

$$
K_S = P+N = \lambda^*K_{\bar{S}} +N,
$$ 

where $\lambda \colon S \rightarrow \bar{S}$ is the contraction to the minimal model $\bar{S}$ of $S$. Then the irreducible curves in the support of $N$ generate a subspace $W \subset H^2(S)$.
\end{definition}

\begin{conj}\label{inter}
$W \cap H^2(X) = \{0\}.$
\end{conj}

\bigskip
\noindent
{\bf Conjecture \ref{inter} implies conjecture \ref{boundc}}:

Assume conjecture \ref{inter} to be true. Then using Noether's inequality it follows:
$$l=\dim L \geq \dim W \geq 2\chi-6-K^2_S=l+2\gamma-6(\chi+1),$$
whence
$$
\gamma \leq 3(\chi+1);
$$
or
$$
c \leq 3(\chi+1)+p_g.
$$

\subsection{The dual surface of a product-quotient surface} \

\noindent
Let $S$ be a product-quotient surface with quotient model 
$$X= (C_1 \times C_2)/G.$$
We assume furthermore that $S$ is {\em regular}, i.e., $q(S) = 0$.
 
Suppose that $S$ is given by a pair of spherical systems of generators: $(a_1, \ldots, a_s)$, $(b_1, \ldots , b_t)$ of $G$.

\begin{defin}
 The {\em dual surface} $S'$ is the product-quotient surface given by the pair of spherical systems of generators: $(a_1, \ldots, a_s)$, $(b_t^{-1}, \ldots , b_1^{-1})$.
\end{defin}

\begin{rem}
It is easy to see that $\frac{1}{n}(1,q) \in \mathfrak{B}(X) \ \iff \frac{1}{n}(1,n-q) \in \mathfrak{B}(X')$.

The invariants we defined in the previous subsection for the dual surface are easily computed. We observe first that 
\begin{itemize}
 \item $\mu(\frac an) = \mu(\frac{n-a}{n})$;
\item $\gamma(\frac an) = - \gamma(\frac{n-a}{n})$.
\end{itemize}
\end{rem}
Denoting as usual by $\gamma, \mu, \ldots $ the invariants of the basket of $S$, and denoting by $\gamma', \mu', \ldots$ or $\gamma(S'), \ldots$ the corresponding invariants of the basket of the dual surfaces $S'$, we can calculate:

$$\gamma = -\gamma', \ \ \mu = \mu',$$
and
$$
\chi(S') = \frac{(g_1-1)(g_2-1)}{|G|} + \frac 14 (\mu-2\gamma') = \frac{(g_1-1)(g_2-1)}{|G|} + \frac 14 (\mu+2\gamma)=
$$
$$
= \chi(S) + \gamma.
$$

In particular, if -as in our situation always- $q(S) =0$, then 
$$
p_g(S') = p_g(S) + \gamma.
$$

\begin{rem}
$(S')' = S$, $p_g(S) = p_g(S') + \gamma'$. 

Moreover, from the proof of proposition \ref{nat} we see that
$$
2p_g(S) = 2(\gamma' + p_g(S')) = h^{1,1}(S') - l' - 2.
$$
Note that $l' = l(S')$.
\end{rem}

We have proven the following:

\begin{prop}
 Let $S$ be q regular product-quotient surface, and denote by $S'$ its dual surface. Then we have for the invariants of $S$ and $S'$ the following relations:
\begin{enumerate}
 \item $\gamma = - \gamma'$;
\item $\mu = \mu '$, $\xi = \xi'$;
\item $\chi(S') = \chi(S) + \gamma$. In particular, if $p_g(S) = 0$ then $h^{1,1}(S') = 2+ l'$, whence $h^{1,1}(X') = 2$, where $X'$ is the singular model of $S'$.
\end{enumerate}

\end{prop}

For the index of $S$ resp. $S'$ we have:
\begin{prop}
$$
\tau(S) := \frac 13 (K_S^2 -2e(S)) = - \frac 13 B(\mathfrak{B}(X)) = -2\gamma -l,
$$

$$
\tau(S') := \frac 13 (K_{S'}^2 -2e(S')) = 2\gamma -l'.
$$
\end{prop}

\begin{rem}
 Let $\bar{S}$ be the minimal model of $S$, then $\tau(S) +(-N^2) = \tau(\bar{S})$. Moreover, by Serrano (cf. \cite{serrano}, we know that for the minimal model of a product-quotient surface, it holds: $\tau(\bar{S}) <0$.

In particular, we get that $l' > 2\gamma$.
\end{rem}

It follows immediately from the above:
$$
\frac 13(B(\mathfrak{B}) + B(\mathfrak{B}')) = l+l' = -(\tau(S) +\tau(S')).
$$

And it is also easy to see that
$$
\frac 13 B(\mathfrak{B}) = l+l'+\tau(S'),
$$
$$
\frac 13 B(\mathfrak{B'}) = l+l'+\tau(S).
$$

\begin{rem}
 Observe that when we go from $S$ to the dual surface $S'$, we consider on $C_1$ the same action of $G$ as for $S$, whereas for $C_2$ we replace the action $y \mapsto g(y)$ by $y \mapsto \overline{g(\overline{y})}$.

Of course we could do the same thing, replacing $y \mapsto g(y)$ by $y \mapsto g \alpha(y)$ for any (holomorphic) automorphism $\alpha$ of $C_2$. Probably we get many new surfaces from this construction (of course depending on the representation theory of $G$).
\end{rem}

\begin{problem}
 Can we use the existence of the dual surface (or a generalized version of the dual surface) in order to prove conjecture \ref{boundc}?
\end{problem}

%%%%%%%%%%%%%%%%%%%%%%%%%%%%%%%%%%%%%%%%%%%%%%%%%%%%%%

\end{document}